\documentclass[letterpaper, 10pt, conference]{ieeeconf}      

\IEEEoverridecommandlockouts                              
\usepackage{amsmath,graphicx,amsfonts,amssymb,epsfig,mathrsfs}
\usepackage{subcaption}
\usepackage{color}
\usepackage{multirow}
\usepackage{rotating}
\usepackage{graphicx}%
\usepackage{algorithm}
\usepackage{algpseudocode}
\usepackage{algorithmicx}
\usepackage{cite}
\usepackage{url}
\usepackage{framed}
\usepackage{bm}

\newtheorem{theorem}{Theorem}
\newtheorem{corollary}[theorem]{Corollary}

\newtheorem{lemma}{Lemma}
\newtheorem{proposition}{Proposition}
\newtheorem{remark}{Remark}

\renewcommand{\det}{{\mathrm{det}}}
\newcommand{\tr}{{\mathrm{trace}}}
\newcommand{\spec}{{\mathrm{spec}}}
\newcommand{\vol}{{\mathrm{vol}}}

\title{\LARGE \bf
On the Parameterized Computation of Minimum Volume Outer Ellipsoid of Minkowski Sum of Ellipsoids
}

\author{Abhishek Halder
\thanks{Abhishek Halder is with the Department of Applied Mathematics and Statistics, University of California, Santa Cruz, CA 95064, USA,
        {\tt\small{ahalder@ucsc.edu}}%
}}

\begin{document}

\maketitle
\thispagestyle{empty}
\pagestyle{empty}

\begin{abstract}
We consider the problem of computing certain parameterized minimum volume outer ellipsoidal (MVOE) approximation of the Minkowski sum of a finite number of ellipsoids. We clarify connections among several parameterizations available in the literature, obtain novel analysis results regarding the conditions of optimality, and based on the same, propose two new algorithms for computing the parameterized MVOE. Numerical results reveal faster runtime for the proposed algorithms than the state-of-the-art semidefinite programming approach of computing the same.
\end{abstract}


\section{Introduction}
The Minkowski sum of two sets $\mathcal{X}$ and $\mathcal{Y}$, which we denote by $\mathcal{Z}=\mathcal{X} \dot{+} \mathcal{Y}$, is the set \begin{eqnarray}
\mathcal{Z}:=\{z \mid z = x+y, x\in\mathcal{X}, y\in\mathcal{Y}\}.
\label{MinkowskiSumDefn}
\end{eqnarray}
The Minkowski sum in general, and the Minknowski sum of ellipsoids in particular, appear frequently in systems, control and robotics applications. As a motivating example, consider computing the reach set of a linear control system:
\begin{eqnarray}
	\bm{x}^{+}(t) = \bm{F}(t)\bm{x}(t) + \bm{G}(t)\bm{u}(t), \quad \bm{x}\in\mathbb{R}^{n}, \quad \bm{u}\in\mathbb{R}^{m},
\label{LTIsys}	
\end{eqnarray}
subject to set-valued uncertainties in its initial conditions $\bm{x}(t_{0}) \in \mathcal{X}_{0}$, or final conditions $\bm{x}(t_{1}) \in \mathcal{X}_{1}$, and control $\bm{u}(t) \in \mathcal{U}(t)$. For continuous time case, $\bm{x}^{+}(t) := \dot{\bm{x}}(t)$, and for discrete time case, $\bm{x}^{+}(t) := \bm{x}(t+1)$. We assume that the sets $\mathcal{X}_{0},\mathcal{X}_{1} \subset \mathbb{R}^{n}$ are compact, and so are the sets $\mathcal{U}(t) \in \mathbb{R}^{m}$ for all $t$. 

Let us denote the \emph{forward} reach set at time $t$ starting from an initial set $\mathcal{X}_{0}$ at time $t_{0} < t$, with feasible control sets $\mathcal{U}(t)$, as $\overrightarrow{\mathcal{R}}\left(\mathcal{X}_{0}, t, t_{0}\right)$. Likewise, denote the \emph{backward} reach set at time $t$ starting from a terminal set $\mathcal{X}_{1}$ at time $t_{1} > t$, with feasible control sets $\mathcal{U}(t)$, as $\overleftarrow{\mathcal{R}}\left(\mathcal{X}_{1}, t, t_{1}\right)$. In words, the forward (resp. backward) reach set at time $t$ is the set of all states that can be achieved at that time via dynamics (\ref{LTIsys}) starting from an initial set $\mathcal{X}_{0}$ (resp. terminal set $\mathcal{X}_{1}$) at time $t_{0}$ (resp. at time $t_{1}$).

\begin{figure}[t]
\centering
\includegraphics[width=.48\textwidth]{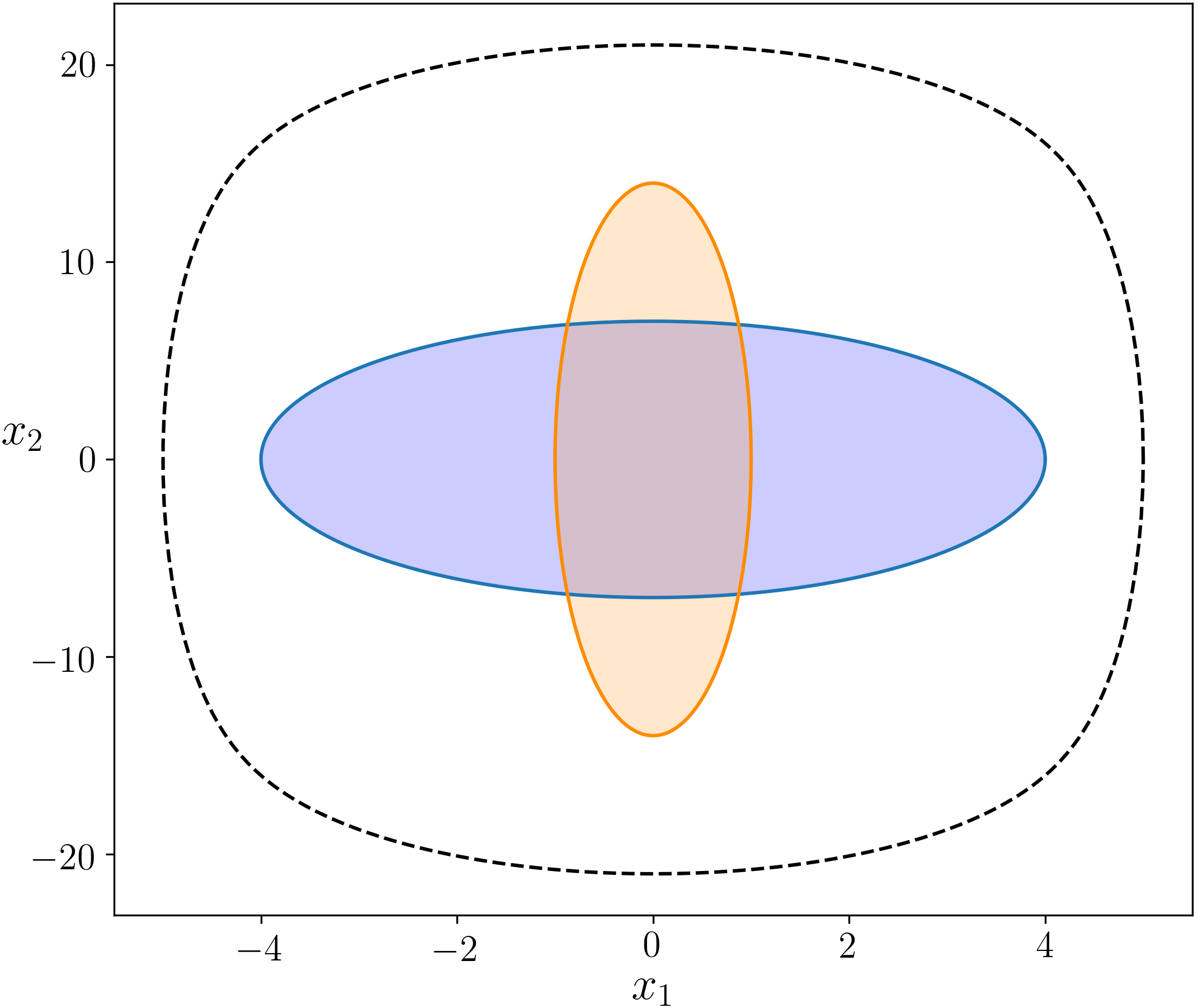}
\caption{The Minkowski sum of two ellipsoids is not an ellipsoid in general, as shown in 2D for two axes-aligned ellipses (\emph{solid} boundaries). The Minkowski sum is the set shown with \emph{dashed} boundary.}
\vspace*{-0.19in}
\label{2DMinkowskiSumEllipseIntroFig}
\end{figure}

In continuous time, we have
{\small{\begin{subequations}
	\begin{align}
		&\overrightarrow{\mathcal{R}}\left(\mathcal{X}_{0}, t, t_{0}\right)	= \bm{\Phi}\left(t,t_{0}\right) \mathcal{X}_{0} \:\dot{+} \displaystyle\int_{t_{0}}^{t} \bm{\Phi}\left(t,\tau\right) \bm{G}(\tau)\: \mathcal{U}(\tau) \:\mathrm{d}\tau, \label{ContTimeForwardReachSet}\\
		&\overleftarrow{\mathcal{R}}\left(\mathcal{X}_{1}, t, t_{1}\right)	= \bm{\Phi}\left(t,t_{1}\right) \mathcal{X}_{1} \:\dot{+} \displaystyle\int_{t_{1}}^{t} \bm{\Phi}\left(t,\tau\right) \bm{G}(\tau)\: \mathcal{U}(\tau) \:\mathrm{d}\tau,\label{ContTimeBackwardReachSet}
	\end{align}
\label{ContTimeReachSets}
\end{subequations}}}
and in discrete time,
{\small{\begin{subequations}
	\begin{align}
		&\overrightarrow{\mathcal{R}}\left(\mathcal{X}_{0},t,t_{0}\right)	= \bm{\Phi}\left(t,t_{0}\right) \mathcal{X}_{0} \:\dot{+}\! \displaystyle\sum_{\tau=t_{0}}^{t-1} \!\bm{\Phi}\left(t,\tau+1\right) \bm{G}(\tau)\: \mathcal{U}(\tau),\label{DiscreteTimeForwardReachSet}\\
		&\overleftarrow{\mathcal{R}}\left(\mathcal{X}_{1}, t, t_{1}\right)	=\bm{\Phi}\left(t,t_{1}\right) \mathcal{X}_{1} \:\dot{+}\displaystyle\sum_{\tau=t}^{t_{1}-1} \!-\bm{\Phi}\left(t,\tau\right) \bm{G}(\tau)\: \mathcal{U}(\tau),
\label{DiscreteTimeBackwardReachSet}
\end{align}
\label{DiscreteTimeReachSets}
\end{subequations}}}
wherein $\bm{\Phi}\left(\cdot,\cdot\right)$ is the state transition matrix associated with (\ref{LTIsys}), and the integrals that follow the Minkowski sums in (\ref{ContTimeReachSets}) are Aumann integrals \cite{Aumann1965}. 

In order to numerically compute the reach sets (\ref{ContTimeReachSets}) or (\ref{DiscreteTimeReachSets}) at any desired time $t$, parametric description of the sets $\mathcal{X}_{0}, \mathcal{X}_{1}$ and $\mathcal{U}(t)$ are sought in practice. It is quite natural to describe the sets $\mathcal{X}_{0}, \mathcal{X}_{1}$ and $\mathcal{U}(t)$ as ellipsoids since they model structured weighted norm bounded uncertainties in initial conditions, terminal conditions and controls, respectively. Given such ellipsoidal set valued description of uncertainties, (\ref{ContTimeReachSets}) requires computing the Minkowski sum of an ellipsoid with an ellipsoidal set valued integral\footnote{the ellipsoidal set valued integral is guaranteed to be convex but may not be an ellipsoid. However, it can be tightly inner and outer approximated by (respectively) unique ellipsoids \cite{john2014extremum,busemann1950foundations}.}, and (\ref{DiscreteTimeReachSets}) requires computing the Minkowski sum of a finite number of ellipsoids. In this vein, a plethora of results have appeared in the systems-control literature based on ellipsoidal calculus \cite{reshetnyak1989summation,KurzhanskiValyi1997,ros2002ellipsoidal,KurzhanskiyVaraiyaCDC2006}.

In robotics, Minkowski sums appear in motion planning problems \cite{LatombeBook} as they quantify the so called ``configuration space obstacle". For safety purposes, it is common \cite{choi2009continuous,best2016real,yan2016path} to engulf the robots and obstacles by ellipsoids before checking collision avoidance. Here, ellipsoidal descriptions are preferred over other convex shapes (e.g. polytopes) for relative computational ease -- an arbitrary ellipsoid in $\mathbb{R}^{n}$ can be parameterized by $\frac{n(n+3)}{2}$ reals\footnote{we need $n$ reals for the center vector, and $\frac{n(n+1)}{2}$ reals for the symmetric shape matrix} describing its center and shape. This fixed parameterization complexity is helpful, for example, in designing communication protocols for multi-agent collision avoidance, where the ellipsoidal descriptions may need to be encoded in communication packets.

While the Minkowski sum of convex sets is convex, the Minkowski sum of ellipsoids is not ellipsoid in general (Fig. \ref{2DMinkowskiSumEllipseIntroFig}). In simple cases like the one shown in Fig. \ref{2DMinkowskiSumEllipseIntroFig}, one can find  explicit formula for the boundary of the Minkowski sum by computing the convolution boundary (see e.g., \cite[Example 2.1]{lee1998polynomial}), but this approach is computationally tedious for higher dimensions and non-axes aligned cases. Recently, a parametric formula for the boundary of the Minkowski sum of two ellipsoids in $\mathbb{R}^{n}$ was obtained in \cite{yan2015closed}. In control and robotics applications, it is common, and for the computational benefits mentioned above, may in fact be desirable, to instead compute a \emph{tight outer ellipsoidal approximation} of the Minkowski sum of ellipsoids. The qualifier ``outer" is motivated by guaranteeing provable safety, while the qualifier ``tight" is motivated by reducing conservatism in the outer approximation, where ``tightness" is promoted by minimizing the size of the outer approximating ellipsoid. Typical measures of size used as optimality criterion \cite{KurzhanskiValyi1997,durieu2001multi} include the sum of the squared semi-axes, and the volume of the ellipsoid. In this paper, we will consider the problem of computing the \emph{minimum volume outer ellipsoid} (MVOE) of the Minkowski sum of ellipsoids. Problems involving MVOE for a given set have appeared before in the context of state estimation with norm-bounded disturbances \cite{schweppe1968recursive,bertsekas1971recursive,schlaepfer1972continuous}, and in the context of system identification \cite{fogel1979system,belforte1990parameter,kosut1992set}.

The purpose of this paper is threefold:
\begin{enumerate}
\item to clarify the connections between several existing results in the literature for an outer ellipsoidal parameterization that contains the Minkowski sum of two given ellipsoids,
\item to provide novel analysis results for the minimum volume condition of optimality,
\item to propose new numerical algorithms based on the above analysis, for computing the MVOE of the Minkowski sum of two ellipsoids.
\end{enumerate}

This paper is organized as follows. In Section II, we collect several existing outer ellipsoidal parameterizations from the literature, which are guaranteed to contain the Minkowski sum of the constituent ellipsoids, show their equivalence and set up the parameterized MVOE problem. Section III contains novel analysis results for the same. In Section IV, we build on the results of Section III, and design two new algorithms for solving the associated MVOE problem. Numerical simulations are given in Section V to elucidate the proposed algorithms. Section VI concludes the paper.

\subsection*{Preliminaries}

\subsubsection{Notations}
$\mathbb{R}^{d}$ stands for the Euclidean $d$-dimensional vector space, and $\mathbb{R}^{d}_{+}$ denotes its positive orthant. We use $\mathcal{B}_{1}^{d}$ to denote the $d$-dimensional Eulcidean unit ball, $\bm{1}$ to denote the vector of ones of appropriate dimension, $\mathbb{N}$ for the set of natural numbers, and $\mathbb{S}_{+}^{d}$ for the cone of real symmetric positive definite matrices of size $d\times d$. Furthermore, $\tr(\cdot)$, $\det(\cdot)$, and $\spec(\cdot)$ respectively denote the trace, determinant and spectrum of a matrix. We use $\vol(\cdot)$ to denote volume, ${\rm{min}}(\cdot,\cdot)$ to denote pointwise minimum, and $\Gamma(\cdot)$ to denote the Gamma function. The notation $f^{n}(x)$ stands for $n$-fold composition of the function $f$ evaluated at $x$, i.e., \[f^{n}(x)\equiv \underbrace{f(f(\hdots f(f}_{\text{$n$ times}}(x))\hdots)).\]

\subsubsection{Ellipsoids}
An ellipsoid with center $\bm{q}\in\mathbb{R}^{d}$ and shape matrix $\bm{Q} \in \mathbb{S}_{+}^{d}$, is denoted by \[\mathcal{E}\left(\bm{q}, \bm{Q}\right) := \{\bm{x} \in \mathbb{R}^{d} : \left(\bm{x} - \bm{q}\right)^{\top}\bm{Q}^{-1}\left(\bm{x} - \bm{q}\right) \leq 1\}.\] The square roots of eigenvalues of $\bm{Q}$ are the lengths of semi-axes of $\mathcal{E}$. Notice that $\bm{Q} \in \mathbb{S}_{+}^{d} \Rightarrow \bm{Q}^{-1},\bm{Q}^{-\frac{1}{2}} \in \mathbb{S}_{+}^{d}$. Let $\bm{N}:=\bm{Q}^{-\frac{1}{2}}$, and let $\bm{L}\bm{L}^{\top}$ be the Cholesky decomposition of $\bm{Q}^{-1}$. Then $\bm{Q}^{-1} = \bm{NN} = \bm{L}\bm{L}^{\top}$. Hence, alternative parameterizations of $\mathcal{E}\left(\bm{q}, \bm{Q}\right)$ are $\mathcal{E}\left(\bm{q}, \bm{N}\right) := \{\bm{x} \in \mathbb{R}^{d} : \parallel\bm{N}\left(\bm{x} - \bm{q}\right)\parallel_{2} \leq 1\}$, and $\mathcal{E}\left(\bm{q}, \bm{L}\right) := \{\bm{x} \in \mathbb{R}^{d} : \parallel\bm{L}^{\top}\left(\bm{x} - \bm{q}\right)\parallel_{2} \leq 1\}$. Another way to express a $d$-dimensional ellipsoid is to view it as the image of an affine transformation of $\mathcal{B}_{1}^{d}$, i.e., $\mathcal{E}\left(\bm{q},\bm{M}\right) := \{\bm{M}\bm{v} + \bm{q} \: : \bm{v}\in\mathbb{R}^{d}, \parallel \bm{v} \parallel_{2} \leq 1\}$, where $\bm{M}:=\bm{Q}^{\frac{1}{2}} = \bm{N}^{-1}$. 

Yet another ellipsoidal parameterization that will appear in the later part of this paper, is via a matrix-vector-scalar triple $(\bm{A},\bm{b},c)$ encoding the quadratic form, i.e., $\mathcal{E}(\bm{A},\bm{b},c) := \{\bm{x} \in \mathbb{R}^{d} : \bm{x}^{\top}\bm{A}\bm{x} + 2\bm{x}^{\top}\bm{b} + c \leq 0\}$. The following relations among $(\bm{A},\bm{b},c)$ and $(\bm{q},\bm{Q})$ parameterizations will be useful:
\begin{eqnarray}
	\bm{A} = \bm{Q}^{-1},\quad \bm{b} = -\bm{Q}^{-1}\bm{q},\quad c = \bm{q}^{\top}\bm{Q}^{-1}\bm{q} - 1,
\label{Qq2Abc}	
\end{eqnarray}
and
\begin{eqnarray}
\bm{Q} = \bm{A}^{-1}, \quad \bm{q} = -\bm{Q}\bm{b}.
\label{Ab2Qq}	
\end{eqnarray}
Furthermore, we have \[\vol\left(\mathcal{E}\left(\bm{q},\bm{Q}\right)\right) = \frac{\vol\left(\mathcal{B}_{1}^{d}\right)}{\sqrt{\det\left(\bm{Q}^{-1}\right)}} = \frac{\pi^{\frac{d}{2}}}{\Gamma\left(\frac{d}{2} + 1\right)}\sqrt{\det\left(\bm{Q}\right)}.\]


\section{Ellipsoid that Contains the Minkowski Sum}
The Minkowski sum of ellipsoids being compact and convex, has unique MVOE \cite{john2014extremum,busemann1950foundations}, known as the \emph{L\"{o}wner-John ellipsoid} $\mathcal{E}_{\rm{LJ}}:=\mathcal{E}(\bm{q}_{\rm{LJ}},\bm{Q}_{\rm{LJ}})$. Specifically, consider $K$ given ellipsoids $\{\mathcal{E}_{k}\}_{k=1}^{K}$ in $\mathbb{R}^{d}$, where $\mathcal{E}_{k} := \mathcal{E}(\bm{q}_{k},\bm{Q}_{k})$. It is easy to see that the MVOE of the Minkowski sum
\begin{eqnarray}
\mathcal{E}_{1} \:\dot{+}\: \mathcal{E}_{2} \:\dot{+}\: \hdots\: \dot{+}\: \mathcal{E}_{K}
\label{MinkSumEllip}	
\end{eqnarray}
has center 
\begin{eqnarray}
\bm{q}_{\rm{LJ}} = \bm{q}_{1} + \bm{q}_{2} + \hdots + \bm{q}_{K}.
\label{centerMinkSum}	
\end{eqnarray}
While no general formula for $\bm{Q}_{\rm{LJ}}$ is known as a function of $\bm{Q}_{1}, \hdots, \bm{Q}_{K}$, it is computationally easier to construct a parameterized family of outer ellipsoids containing the Minkowski sum by first constructing certain parametric function of $\bm{Q}_{1}, \hdots, \bm{Q}_{K}$, and then optimizing over the parameter. In fact, one can find a parameterization that is known to be inclusion minimal external estimate of (\ref{MinkSumEllip}) (see \cite[p. 112, Thm. 2.2.1]{KurzhanskiValyi1997}). In the following, we collect such parameterizations that have appeared in the literature, and show their equivalence.

\subsection{Equivalent Parameterizations}
An outer parameterization of the shape matrix $\bm{Q}$, such that the corresponding ellipsoid is guaranteed to contain the Minkowski sum (\ref{MinkSumEllip}) with respective shape matrices $\bm{Q}_{1}, \hdots, \bm{Q}_{K}$, is given by the Kurzhanski parameterization \cite{KurzhanskiValyi1997,kurzhanski2002reachability}:
\begin{eqnarray}
\bm{Q}(\bm{\ell}) = \left(\displaystyle\sum_{k=1}^{K} \sqrt{\bm{\ell}^{\top}\bm{Q}_{k}\bm{\ell}}\right) \displaystyle\sum_{k=1}^{K}\frac{\bm{Q}_{k}}{\sqrt{\bm{\ell}^{\top}\bm{Q}_{k}\bm{\ell}}},
\label{KurzhanskyQl}	
\end{eqnarray}
where the parameterization variable is unit vector $\bm{\ell}\in\mathbb{R}^{d}$, $\bm{\ell}^{\top}\bm{\ell}=1$.

Durieu, Walter and Polyak \cite{durieu2001multi} proposed a similar parameterization: 
\begin{eqnarray}
\bm{Q}(\bm{\alpha}) = \displaystyle\sum_{k=1}^{K} \alpha_{k}^{-1} \bm{Q}_{k}
\label{DWPQalpha}	
\end{eqnarray}
where $\bm{\alpha}\in\mathbb{R}^{K}_{+}$, $\bm{1}^{\top}\bm{\alpha}=1$.

Thanks to the transitive nature of these parameterizations, the optimal parameterization in some prescribed sense, can be found through a pairwise recursion over the constituent ellipsoids (see e.g., recursion (44) in \cite{durieu2001multi}). Therefore, it suffices to consider the $K=2$ case.

In \cite{schweppe1973,Chernousko1980PartI,MaksarovNorton1996,KurzhanskiValyi1997}, the following parameterization for $K=2$ was given:
\begin{eqnarray}
\bm{Q}\left(\beta\right) = \left(1 + \frac{1}{\beta}\right)\bm{Q}_{1} \: + \: \left(1 + \beta\right)\bm{Q}_{2}.
\label{SchweppeKVQbeta}	
\end{eqnarray}
We notice that the parameterization (\ref{SchweppeKVQbeta}) can be obtained from parameterization (\ref{KurzhanskyQl}) via the transformation 
\begin{eqnarray}
\sqrt{\displaystyle\frac{\bm{\ell}^{\top}\bm{Q}_{1}\bm{\ell}}{\bm{\ell}^{\top}\bm{Q}_{2}\bm{\ell}}} \mapsto \beta.
\label{betaFromell}
\end{eqnarray}
Furthermore, we can get parameterization (\ref{SchweppeKVQbeta}) from parameterization (\ref{DWPQalpha}) by setting
\begin{eqnarray}
\alpha_{2} = 1 - \alpha_{1}, \quad \displaystyle\frac{\alpha_{1}}{1 - \alpha_{1}} \mapsto \beta .
\label{betaFromalpha}	
\end{eqnarray}

\subsection{Minimum Volume Parametric Optimization}
Given a scalar $\beta>0$, and a pair of matrices $\bm{Q}_{1},\bm{Q}_{2}\in\mathbb{S}^{d}_{+}$, let $\bm{Q}\left(\beta\right) := \left(1 + \frac{1}{\beta}\right)\bm{Q}_{1} \: + \: \left(1 + \beta\right)\bm{Q}_{2}$, as in (\ref{SchweppeKVQbeta}). Clearly, $\bm{Q}(\beta)\in\mathbb{S}^{d}_{+}$. In the following, we will study the parametric optimization problem 
\begin{eqnarray}
\underset{\beta>0}{\text{minimize}} \: \log\det\left(\bm{Q}\left(\beta\right)\right)
\label{SuperEllipsoidForMVE2}	
\end{eqnarray}
that corresponds to the minimum volume criterion. 

We mention here that instead of minimizing the volume, if one minimizes the sum of squared semi-axes lengths (which amounts to replacing the objective function in (\ref{SuperEllipsoidForMVE2}) by $\tr(\bm{Q}\left(\beta\right))$), then the optimal $\beta>0$ can be found analytically:
\[\beta = \sqrt{\displaystyle\frac{\tr\left(\bm{Q}_{1}\right)}{\tr\left(\bm{Q}_{2}\right)}}.\]  
Next, we analyze the optimality conditions for (\ref{SuperEllipsoidForMVE2}).


\section{Analysis}

\subsection{Optimality Condition}
Letting $\bm{R} := \bm{Q}_{1}^{-1}\bm{Q}_{2}$, notice from (\ref{SchweppeKVQbeta}) that
\begin{eqnarray}
&&\frac{\partial}{\partial\beta}\bm{Q}(\beta) = -\frac{1}{\beta^{2}}\bm{Q}_{1}\left(\bm{I} - \beta^{2}\bm{R}\right), \\
&&(\bm{Q}(\beta))^{-1} = \frac{\beta}{1+\beta}\left(\bm{I} + \beta\bm{R}\right)^{-1}\bm{Q}_{1}^{-1},
\label{DerivativeAndInverseOfQ}	
\end{eqnarray}
and we thus get
\begin{eqnarray}
\displaystyle\frac{\partial}{\partial\beta}\log\det\left(\bm{Q}(\beta)\right) = \tr\left((\bm{Q}(\beta))^{-1}\displaystyle\frac{\partial}{\partial\beta}\bm{Q}(\beta)\right)\nonumber\\
= -\displaystyle\frac{1}{\beta(1+\beta)}\:\tr\left(\left(\bm{I} + \beta\bm{R}\right)^{-1} \left(\bm{I} - \beta^{2}\bm{R}\right)\right).
\label{FOOC}	
\end{eqnarray}
To proceed further, we need the following lemma.
\begin{lemma}
Given symmetric matrices $\bm{M}_{1}$ and $\bm{M}_{2}$, if $\bm{M}_{1}$ is positive definite, then the product $\bm{M}_{1}\bm{M}_{2}$ is diagonalizable.
\label{AisSPDandBisS}	
\end{lemma}
\begin{proof}
Since $\bm{M}_{1}\in\mathbb{S}_{+}^{d}$, there exists a unique matrix $\bm{N}_{1}\in\mathbb{S}_{+}^{d}$ such that $\bm{N}_{1}\bm{N}_{1} = \bm{M}_{1} \Leftrightarrow \bm{N}_{1} = \bm{M}_{1}^{\frac{1}{2}}$. In words, $\bm{N}_{1}$ is the unique symmetric positive definite square root of $\bm{M}_{1}$. Now observe that 
\begin{eqnarray*}
\bm{N}_{1}^{-1}\bm{M}_{1}\bm{M}_{2}\bm{N}_{1} = 	\bm{M}_{1}^{\frac{1}{2}}\bm{M}_{2}\bm{M}_{1}^{\frac{1}{2}},
\end{eqnarray*}
where the right-hand-side is symmetric since both $\bm{M}_{1}^{\frac{1}{2}}$ and $\bm{M}_{2}$ are symmetric, thereby demonstrating that $\bm{M}_{1}\bm{M}_{2}$ is similar to a symmetric matrix, and hence diagonalizable.	
\end{proof}

Lemma \ref{AisSPDandBisS} has the following consequence.

\begin{proposition}
The matrix $\bm{R}$ is diagonalizable. 
\label{RisDiagonalizable}	
\end{proposition}
\begin{proof}
Notice that $\bm{Q}_{1} \in \mathbb{S}^{d}_{+} \Rightarrow \bm{Q}_{1}^{-1} \in \mathbb{S}^{d}_{+}$. Then by Lemma \ref{AisSPDandBisS}, the matrix $\bm{R} := \bm{Q}_{1}^{-1}\bm{Q}_{2}$ is diagonalizable.
\end{proof}

Thanks to Proposition \ref{RisDiagonalizable}, there exist nonsingular matrix $\bm{S}$ and diagonal matrix $\bm{\Lambda}$, such that $\bm{R} = \bm{S}\bm{\Lambda}\bm{S}^{-1}$. Furthermore, the diagonal entries of $\bm{\Lambda}$, denoted as $\lambda_{i}$, $i=1,\hdots,d$, are all positive since 
\begin{eqnarray*}
\{\lambda_{i}\}_{i=1}^{d}  = \spec\left(\bm{R}\right)  &=& \spec\left(\bm{Q}_{1}^{-1}\bm{Q}_{2}^{\frac{1}{2}}\bm{Q}_{2}^{\frac{1}{2}}\right) \nonumber\\
&=&\spec\left(\bm{Q}_{2}^{\frac{1}{2}}\bm{Q}_{1}^{-\frac{1}{2}}\bm{Q}_{1}^{-\frac{1}{2}}\bm{Q}_{2}^{\frac{1}{2}}\right),	
\end{eqnarray*}
where the last step follows from the fact that the spectrum of product of two matrices of same size, remains invariant under the change in order of their multiplication (Theorem 1.3.22 in \cite{HornJohnson2ndEd}).

Substituting $\bm{R} = \bm{S}\bm{\Lambda}\bm{S}^{-1}$ and $\bm{I}=\bm{S}\bm{S}^{-1}$ in (\ref{FOOC}), and using the invariance of trace of a matrix product under cyclic permutation, the first order optimality condition $\frac{\partial}{\partial\beta}\log\det(\bm{Q}(\beta)) = 0$ results the following nonlinear algebraic equation:
\begin{eqnarray}
\displaystyle\sum_{i=1}^{d} \displaystyle\frac{1 - \beta^{2}\lambda_{i}}{1 + \beta\lambda_{i}} = 0,
\label{DerivativeEqualsZero}	
\end{eqnarray}
to be solved for $\beta>0$, with known parameters $\lambda_{i}>0$, $i=1,\hdots,d$. 

If there exists a unique positive root of (\ref{DerivativeEqualsZero}), denoted as $\beta_{+}$, then it would indeed correspond to a minimum for problem (\ref{SuperEllipsoidForMVE2}) since
\begin{eqnarray}
&&\frac{\partial^{2}}{\partial\beta^{2}}\log\det\left(\bm{Q}(\beta)\right) \bigg\rvert_{\beta=\beta_{+}} \nonumber\\
&&= \frac{1}{\beta_{+}(1+\beta_{+})} \sum_{i=1}^{d} \frac{\beta_{+}^{2}\lambda_{i}^{2} + (1+2\beta_{+})\lambda_{i}}{(1 + \beta_{+}\lambda_{i})^{2}} > 0.\quad\qquad
\label{SecondDerivativePositive}
\end{eqnarray}
\noindent That the right-hand-side of (\ref{SecondDerivativePositive}) is positive follows from the fact that both $\beta_{+}$ and $\lambda_{i}$ are positive for all $i=1,\dots,d$. 

\begin{remark}
The algebraic equation (\ref{DerivativeEqualsZero}) we derived is different but consistent with another algebraic equation derived in Appendix A.1 of \cite{MaksarovNorton1996} (see equation (A.11) therein) for the first order optimality condition corresponding to (\ref{SuperEllipsoidForMVE2}). To see the consistency, notice that equation (A.9) in \cite{MaksarovNorton1996} can be re-written in our notation as 
\begin{eqnarray}
\tr\left(\left(\bm{I} + \beta\bm{R}\right)^{-1} \bm{R}\right) = \sum_{i=1}^{d}\frac{\lambda_{i}}{1 + \beta\lambda_{i}} = \frac{d}{\beta(\beta+1)},	
\label{RelatingA11withUs}
\end{eqnarray}
which after partial fraction expansion in $\lambda_{i}$, and using the fact (Theorem 1.3.22 in \cite{HornJohnson2ndEd}) that $\spec(\bm{R}) = \spec(\bm{Q}_{2}\bm{Q}_{1}^{-1})$, results (A.11) in \cite{MaksarovNorton1996}. Combining (\ref{RelatingA11withUs}) above with (A.11) in \cite{MaksarovNorton1996}, indeed results (\ref{DerivativeEqualsZero}). The authors in \cite{MaksarovNorton1996} indirectly argue that (A.11) therein admits unique positive solution by referring to \cite{Chernousko1980PartI}. In this paper, we will instead focus on solving (\ref{DerivativeEqualsZero}) and present numerical algorithms for the same.
\end{remark}

\subsection{Uniqueness of $\beta_{+}$}
Except the trivial case of $d=1$, it is not obvious that (\ref{DerivativeEqualsZero}) admits unique positive root. In the following, we will establish the uniqueness of the positive root for any $d\in\mathbb{N}$. 

For $\beta>0$, we can rewrite (\ref{DerivativeEqualsZero}) as an $(d+1)$\textsuperscript{th} degree polynomial in $\beta$:
\begin{eqnarray}
\displaystyle\sum_{i=1}^{d} \left(1 - \beta^{2}\lambda_{i}\right) \displaystyle\prod_{\stackrel{j=1}{j\neq i}}^{d}\left(1 + \beta\lambda_{j}\right) = 0,\nonumber\\
\Leftrightarrow p_{d+1}(\beta) := \displaystyle\sum_{i=1}^{d} \left(\beta^{2}\lambda_{i} - 1\right) \displaystyle\prod_{\stackrel{j=1}{j\neq i}}^{d}\left(\beta\lambda_{j}+1\right) = 0.
\label{Polynomial}	
\end{eqnarray}
Since $\lambda_{i}>0$ for all $i=1,\hdots,d$, hence $\prod_{i=1}^{d}\lambda_{i}>0$. Dividing both sides of (\ref{Polynomial}) by $\prod_{i=1}^{d}\lambda_{i}$, we then get 
\begin{eqnarray}
\displaystyle\sum_{i=1}^{d}\left(\beta^{2}-\frac{1}{\lambda_{i}}\right)\displaystyle\prod_{\stackrel{j=1}{j\neq i}}^{d}\left(\beta + \frac{1}{\lambda_{j}}\right) = 0.
\label{AfterDivisionBylambdaprod}	
\end{eqnarray}
Let us now define
\begin{eqnarray}
\mu_{r} := \left(d-r\right)e_{r} - \left(r-1\right)e_{r-1}, \quad r=1,\hdots,d-1,
\label{mur}	
\end{eqnarray}
where $e_{r} \equiv e_{r}\left(\frac{1}{\lambda_{1}},\hdots,\frac{1}{\lambda_{d}}\right)$ for $r=1,\hdots,d$, denotes the $r$\textsuperscript{th} \emph{elementary symmetric polynomial} \cite[Ch. 2.22]{HardyLittlewoodPolyaBook} in variables $\frac{1}{\lambda_{1}},\hdots,\frac{1}{\lambda_{d}}$. Specifically, 
\begin{eqnarray*}
e_{r} \equiv e_{r}\left(\frac{1}{\lambda_{1}},\hdots,\frac{1}{\lambda_{d}}\right) \!:=\!\! \displaystyle\sum_{1 \leq i_{1} < i_{2} < \hdots < i_{r} \leq d} \frac{1}{\lambda_{i_{1}}\hdots\lambda_{i_{r}}}.	
\end{eqnarray*}
For example, 
\begin{eqnarray*}
e_{1} = \!\!\displaystyle\sum_{1\leq i \leq d}\!\!\lambda_{i}^{\!-1}, \quad e_{2} = \!\!\!\displaystyle\sum_{1\leq i < j \leq d}\!\!\!\!\left(\lambda_{i}\lambda_{j}\right)^{\!-1},\quad
e_{d} = \!\left(\prod_{1\leq i \leq d}\!\!\lambda_{i}\!\right)^{\!\!\!-1}\!\!\!\!,
\end{eqnarray*}
and $e_{0}=1$ by convention. Notice that (\ref{AfterDivisionBylambdaprod}) can be written in the expanded form
\begin{align}
d\beta^{d+1} &+ \left(d-1\right)e_{1}\beta^{d} + \left(\displaystyle\sum_{r=1}^{d-1}\mu_{r}\beta^{d-r+1}\right) \nonumber\\
&-(d-1)e_{d-1}\beta - de_{d} = 0,
\label{ExpandedForm}	
\end{align}
and that $e_{r} > 0$ for all $r=1,\hdots,d$, since $\{\lambda_{i}\}_{i=1}^{d}>0$.

\begin{lemma}\label{muStrictlyIncreasingSeq}
The sequence $\{\mu_{r}\}_{r=1}^{d-1}$ is strictly increasing.	
\end{lemma}
\begin{proof}
For $r=1$, using (\ref{mur}) we have 
\begin{eqnarray*}
\mu_{2} - \mu_{1} = (d-2)(e_{2}+e_{1})>0.	
\end{eqnarray*}
Since $2r>r+1$ for all $r\geq 2$, therefore (using (\ref{mur}) again)
\begin{eqnarray*}
&&\mu_{r+1} - \mu_{r} \nonumber\\
&=&\left(d-r-1\right)e_{r+1} + \left(d-2r\right)e_{r} + \left(r-1\right)e_{r-1}\nonumber\\
&=&\left(d-2r\right)\left(e_{r+1}+e_{r}\right) + \left(r-1\right)\left(e_{r+1}+e_{r-1}\right) > 0,
\end{eqnarray*}
as each of the four parenthetical terms above are positive.	
\end{proof}
\begin{corollary}\label{AllPosmucoeff}
The coefficients $\{\mu_{r}\}_{r=1}^{d-1}$ in (\ref{ExpandedForm}) are all positive.	
\end{corollary}
\begin{proof}
Combining Lemma \ref{muStrictlyIncreasingSeq} and that $\mu_{1} = (d-1)e_{1} > 0$, yields the ordering $\mu_{d-1} > \hdots > \mu_{2} > \mu_{1} > 0$. Hence the statement.	
\end{proof}
We are now ready to demonstrate the uniqueness of $\beta_{+}$.
\begin{theorem}\label{UniquenessThm}
The polynomial equation (\ref{ExpandedForm}) (which is equivalent to (\ref{AfterDivisionBylambdaprod}) or (\ref{Polynomial})) has unique positive root $\beta_{+}$.	
\end{theorem}
\begin{proof}
From Corollary \ref{AllPosmucoeff}, we observe that 	only the last two terms (i.e., linear in $\beta$ and constant term) of the $(d+1)$\textsuperscript{th} degree polynomial (\ref{ExpandedForm}) have negative coefficients while all the preceding terms have positive coefficients. In other words, the sequence of coefficients of (\ref{ExpandedForm}) undergoes only one change in signs: from the positive coefficient of $\beta^{2}$ to the negative coefficient of $\beta$. Therefore by Descartes' rule of sign, (\ref{ExpandedForm}) has unique positive root $\beta_{+}$.
\end{proof}

\begin{figure*}[t!]
    \centering
    \begin{subfigure}[t]{0.48\textwidth}
        \centering
        \includegraphics[width=0.9\textwidth]{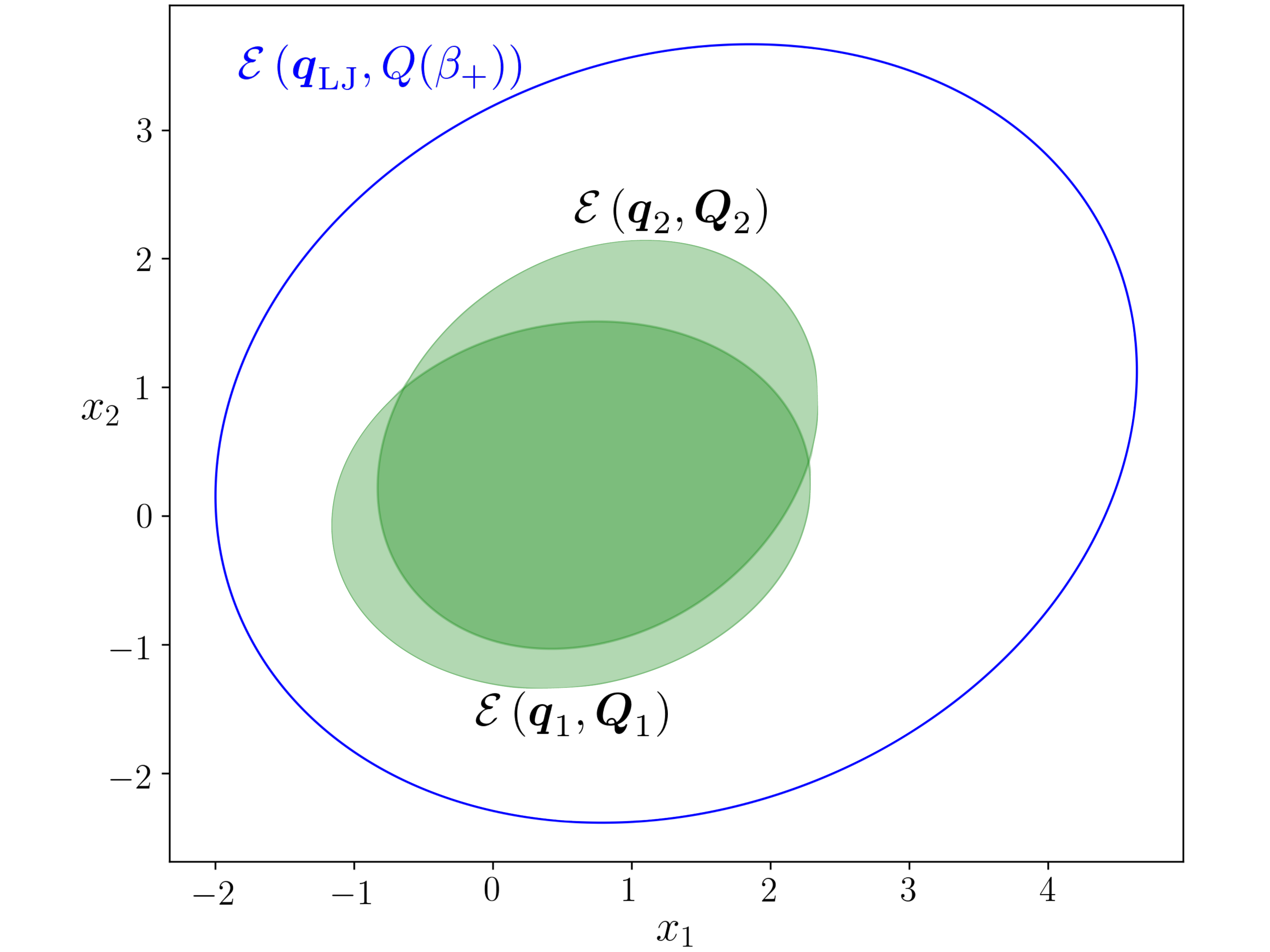}
        \caption{Given two ellipses $\mathcal{E}\left(\bm{q}_{1},\bm{Q}_{1}\right)$ and $\mathcal{E}\left(\bm{q}_{2},\bm{Q}_{2}\right)$ (shown in \emph{green}), we compute the optimal parameterized MVOE $\mathcal{E}\left(\bm{q}_{\rm{LJ}},\bm{Q}(\beta_{+})\right)$ (shown in \emph{blue}) containing the Minkowski sum $\mathcal{E}\left(\bm{q}_{1},\bm{Q}_{1}\right) \dot{+} \mathcal{E}\left(\bm{q}_{2},\bm{Q}_{2}\right)$ using the algorithm given in Section IV.A.1.}
    \end{subfigure}%
    ~ 
    \begin{subfigure}[t]{0.48\textwidth}
        \centering
        \includegraphics[width=0.99\textwidth]{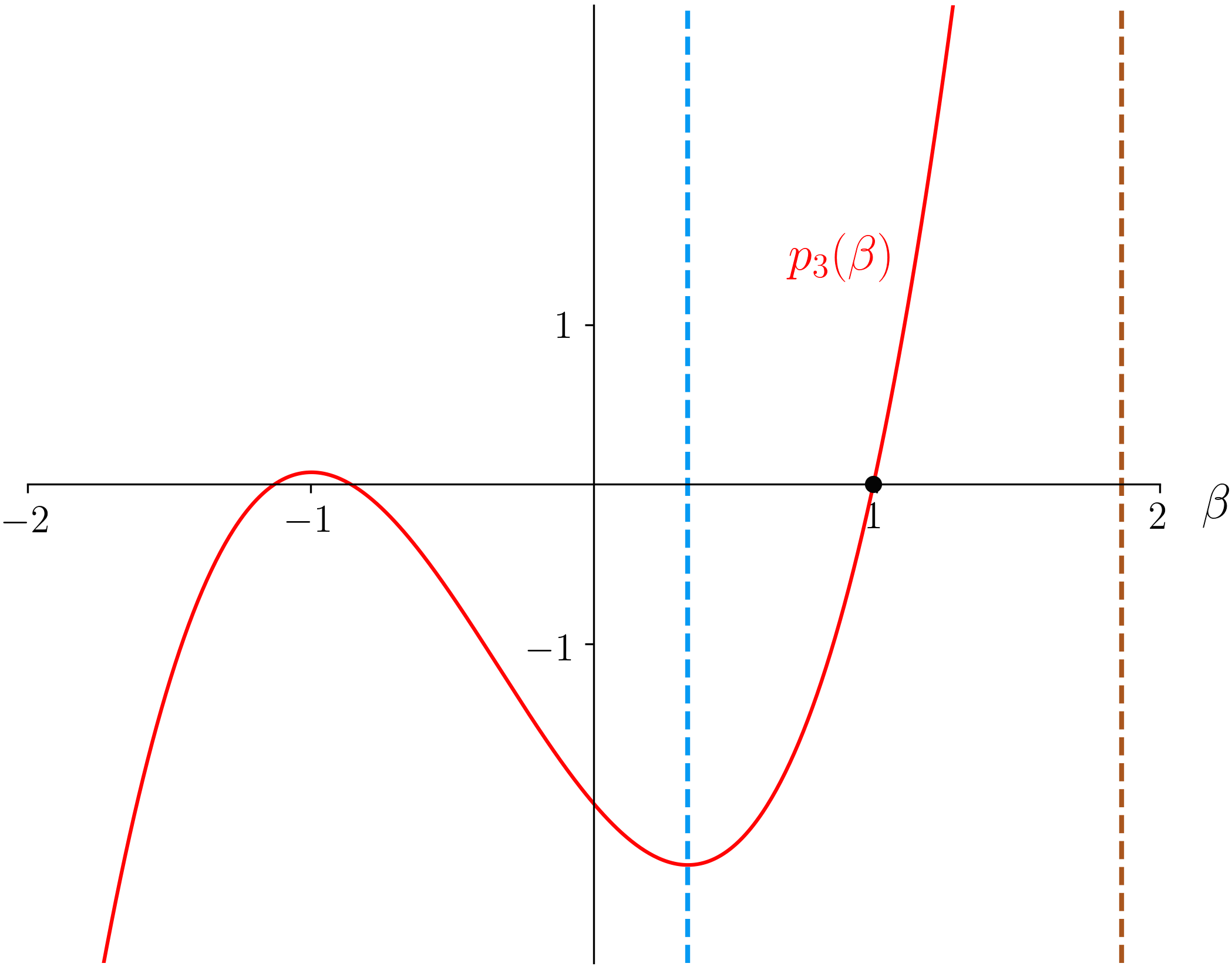}
        \caption{Plot of the function $p_{3}(\beta)$ (in \emph{red}) versus $\beta$, described in Section IV.A.1, for the input ellipses $\mathcal{E}\left(\bm{q}_{1},\bm{Q}_{1}\right)$ and $\mathcal{E}\left(\bm{q}_{2},\bm{Q}_{2}\right)$ in subfigure (a). The \emph{dark circle} is the unique positive root $\beta_{+}$ computed via bisection method using the upper and lower bounds from (\ref{BracketRootPlanarCase}). These bounds are shown above as the \emph{dashed lines}.}
    \end{subfigure}
    \caption{Numerical example depicting the root bracketing algorithm given in Section IV.A.1 for computing $\beta_{+}$.}
\label{RootBracketingFigure}    
\end{figure*}


\section{Algorithms}
In this Section, we present numerical algorithms to solve (\ref{DerivativeEqualsZero}). For the planar case ($d=2$), we present a simple root-bracketing algorithm. For the general case ($d>2$), we derive a fixed point recursion.

\subsection{Bracketing $\beta_{+}$}

\subsubsection{The Planar Case}
Specializing (\ref{Polynomial}) for $d=2$ results a cubic equation
\begin{eqnarray*}
p_{3}(\beta) = 2\lambda_{1}\lambda_{2} \beta^3 + \left(\lambda_{1}+\lambda_{2}\right) \beta^2 - \left(\lambda_{1}+\lambda_{2}\right) \beta - 2 = 0,
\label{PlanarPolynomial}	
\end{eqnarray*}
which by Theorem \ref{UniquenessThm} (alternatively, by directly applying Descartes' rule of sign), has unique positive root $\beta_{+}$. While an explicit expression for $\beta_{+}$ as a function of $\lambda_{1}$ and $\lambda_{2}$ is unwieldy, we next show simple calculations that allow us to bracket the root $\beta_{+}$, thereby facilitating the use of numerical algorithms such as bisection or Newton's method to locate it. To this end, notice that $p_{3}(0) = -2$, $p_{3}^{\prime}(0)=-(\lambda_{1}+\lambda_{2})<0$, $p_{3}^{\prime\prime}(0)=2(\lambda_{1}+\lambda_{2})>0$, which imply that at $\beta=0$, the graph of $p_{3}(\beta)$ is decreasing and concave up; so the root $\beta_{+}$ must be greater than 
\begin{eqnarray*}
\dfrac{\sqrt{(\lambda_{1}+\lambda_{2})(\lambda_{1}+\lambda_{2}+6\lambda_{1}\lambda_{2})}-(\lambda_{1}+\lambda_{2})}{6\lambda_{1}\lambda_{2}},	
\end{eqnarray*}
which is the abscissa of the minimum of $p_{3}(\beta)$. On the other hand, for $i,j=1,2$, setting $\lambda_{i}$ equal to zero, reduces $p_{3}(\beta)$ to a parabola with positive zero 
\begin{eqnarray*}
\dfrac{\lambda_{j}+\sqrt{\lambda_{j}(\lambda_{j}+8)}}{2\lambda_{j}}, \quad {\text{where}}\quad j\neq i.	
\end{eqnarray*}
As a result, we have
\begin{eqnarray}
\dfrac{\sqrt{(\lambda_{1}+\lambda_{2})(\lambda_{1}+\lambda_{2}+6\lambda_{1}\lambda_{2})}-(\lambda_{1}+\lambda_{2})}{6\lambda_{1}\lambda_{2}} < \beta_{+} < \nonumber\\
\min\left\{\dfrac{\lambda_{1}+\sqrt{\lambda_{1}(\lambda_{1}+8)}}{2\lambda_{1}},\dfrac{\lambda_{2}+\sqrt{\lambda_{2}(\lambda_{2}+8)}}{2\lambda_{2}}\right\}.
\label{BracketRootPlanarCase}	
\end{eqnarray}
In Fig. \ref{RootBracketingFigure}, for two constituent ellipses
$\mathcal{E}\left(\bm{q}_{1},\bm{Q}_{1}\right)$ and $\mathcal{E}\left(\bm{q}_{2},\bm{Q}_{2}\right)$ (in \emph{green}, in Fig. \ref{RootBracketingFigure}(a)), we illustrate the optimal parameterized MVOE $\mathcal{E}\left(\bm{q}_{\rm{LJ}},\bm{Q}(\beta_{+})\right)$ (in \emph{blue}, in Fig. \ref{RootBracketingFigure}(a)) containing the Minkowski sum $\mathcal{E}\left(\bm{q}_{1},\bm{Q}_{1}\right) \dot{+} \mathcal{E}\left(\bm{q}_{2},\bm{Q}_{2}\right)$, wherein $\beta_{+}$ is computed via bisection method using the bounds given in (\ref{BracketRootPlanarCase}). The computation of $\beta^{+}$ is depicted in Fig. \ref{RootBracketingFigure}(b).

\subsubsection{The General Case}
It is evident that as $d$ becomes large, generalizing the above approach becomes intractable for higher degree polynomial $p_{d+1}(\beta)$. To circumvent this issue, we next present a fixed point iteration algorithm with guaranteed convergence to $\beta_{+}$.

\subsection{Fixed Point Iteration}
Rewriting the first order optimality condition (\ref{DerivativeEqualsZero}) as 
\begin{eqnarray*}
\beta^{2}\sum_{i=1}^{d}\frac{\lambda_{i}}{1+\beta\lambda_{i}} = \sum_{i=1}^{d}\frac{1}{1+\beta\lambda_{i}},	
\end{eqnarray*}
we consider the following fixed point iteration:
\begin{eqnarray}
\beta_{n+1} = g\left(\beta_{n}\right) := \left(\displaystyle\frac{\sum_{i=1}^{d} \frac{1}{1 + \beta_{n}\lambda_{i}}}{\sum_{i=1}^{d} \frac{\lambda_{i}}{1 + \beta_{n}\lambda_{i}}}\right)^{\!\frac{1}{2}},
\label{FixedPointIteration}	
\end{eqnarray}
where $g : \mathbb{R}_{+} \mapsto \mathbb{R}_{+}$, i.e., $g$ is cone-preserving. By harnessing the nonlinear Perron-Frobenius theory for cone preserving maps \cite{LemmensNussbaumBook2012,KrauseBook2015}, the following theorem ensures that the iteration (\ref{FixedPointIteration}) indeed converges to $\beta_{+}$.
\begin{theorem}
Starting from any initial guess $\beta_{0} \in \mathbb{R}_{+}$, the iteration (\ref{FixedPointIteration}) converges to a unique fixed point $\beta_{+}\in\mathbb{R}_{+}$, i.e., $\displaystyle\lim_{n\rightarrow\infty}g^{n}(\beta_{0})=\beta_{+}$.
\end{theorem}

\begin{proof}
We know that $g$ is cone preserving. For $\lambda_{i},x>0$, consider the positive convex functions $f_{i}:=\frac{1}{1+x\lambda_{i}}$, and let
\[\phi(x) := \sqrt{x}, \quad \text{and} \quad \psi(x) := \frac{\sum_{i}f_{i}}{\sum_{i}\lambda_{i}f_{i}}.\] 
It is not difficult to show that both $\phi(x)$ and $\psi(x)$ are concave and increasing, and hence \cite[p. 84]{BoydCvxBook} so is $g(\beta_{n}) = \phi(\psi(\beta_{n}))$ as a function of $\beta_{n}$. Consequently (see the first step in the proof of Theorem 2.1.11 in \cite{KrauseBook2015}) $g$ is contractive in Hilbert metric on the cone $\mathbb{R}_{+}$. By Banach contraction mapping theorem, $g$ admits unique fixed point $\beta_{+} \in \mathbb{R}_{+}$ and $\displaystyle\lim_{n\rightarrow\infty}g^{n}(\beta_{0})=\beta_{+}$.
\end{proof}


\begin{figure}[t]
\centering
\includegraphics[width=.45\textwidth]{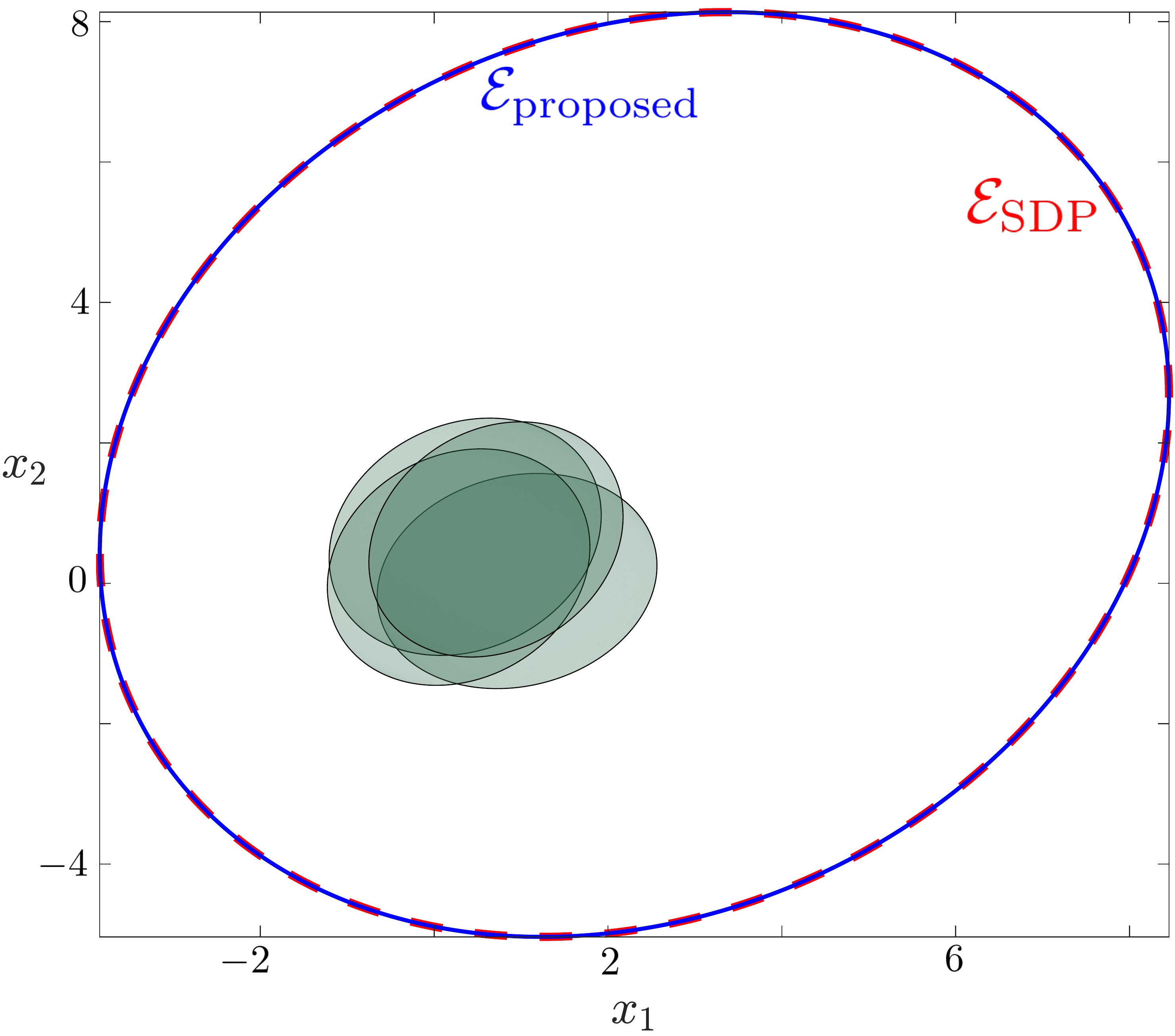}
\caption{The parameterized MVOE $\mathcal{E}_{\text{proposed}}$ (\emph{blue solid line}) is computed for the Minkowski sum of $K=4$ randomly generated ellipses (\emph{green filled}), by recursively applying the algorithm proposed for a pair in Section IV.A.1. The solution $\mathcal{E}_{\text{proposed}}$ matches with the ellipse $\mathcal{E}_{\text{SDP}}$ (\emph{red dashed line}) computed by solving (\ref{BoydSDP})-(\ref{Constr}) via {\texttt{cvx}}.}
\vspace*{-0.15in}
\label{2Dcomparison}
\end{figure}

\begin{figure}[t]
\centering
\includegraphics[width=.48\textwidth]{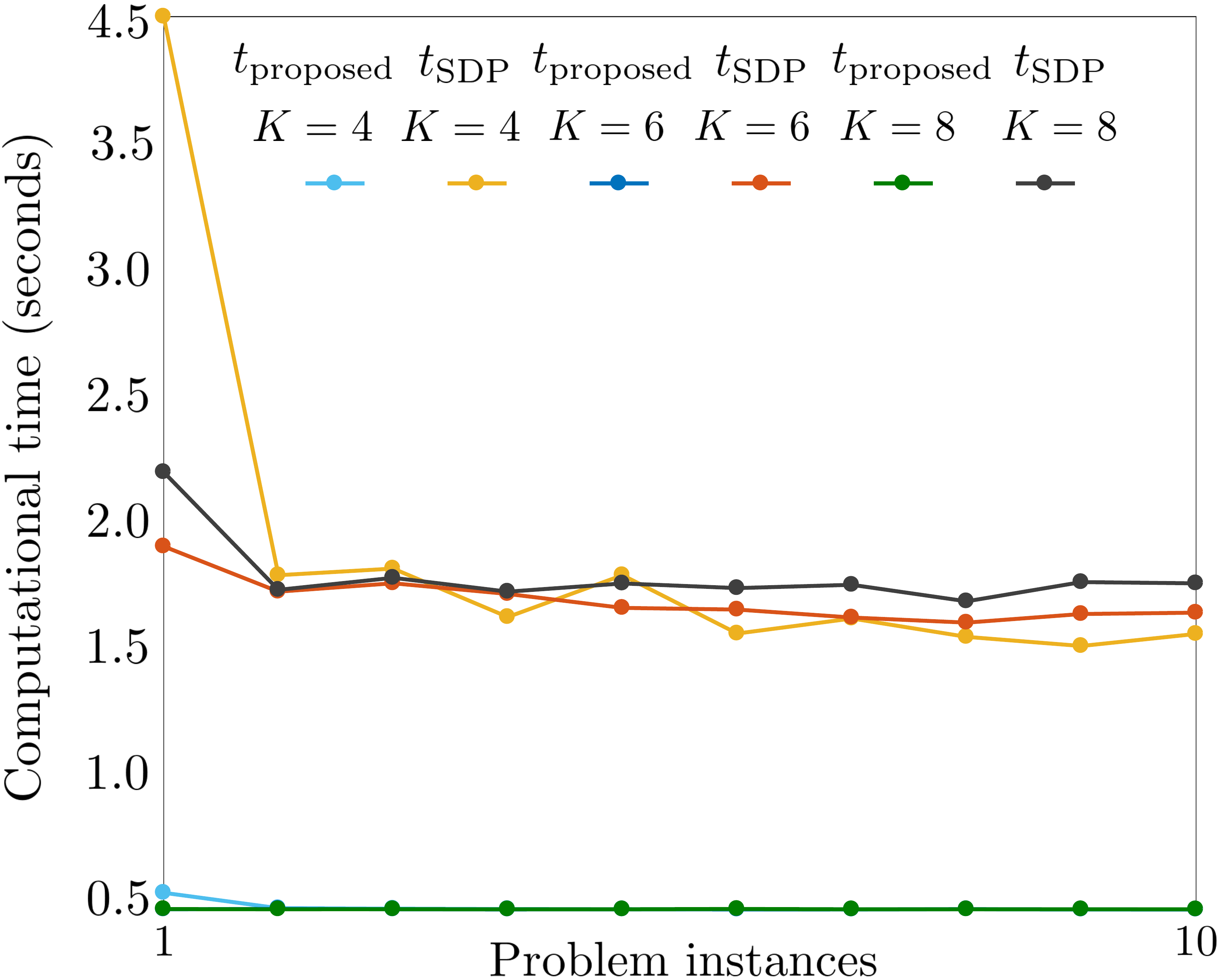}
\vspace*{-0.09in}
\caption{Comparison of computational times $t_{\text{proposed}}$ and $t_{\text{SDP}}$, for the proposed algorithm (Section IV.A.1) and the SDP (\ref{BoydSDP})-(\ref{Constr}) respectively, in solving the 2D MVOE problems for the Minkowski sum of $K=4,6,8$ random ellipses. The above results are for 10 different random problem instances, for each fixed $K$. In all cases, $t_{\text{proposed}} << t_{\text{SDP}}$.}
\vspace*{-0.13in}
\label{2Dstat}
\end{figure}

\section{Numerical Simulations}
In this Section, we will compare the computational performance of the  algorithms proposed in Section IV, with the current state-of-the-art, which is to reformulate the problem of computing MVOE of the Minkowski sum of a given set of ellipsoids as a semi-definite programming (SDP) problem via the $\mathcal{S}$-procedure (see e.g., \cite[Ch. 3.7.4]{BoydLMIBook}). Specifically, given $K$ constituent ellipsoids $\mathcal{E}(\bm{q}_{i},\bm{Q}_{i})$ or equivalently $\mathcal{E}(\bm{A}_{i},\bm{b}_{i},c_{i})$ in $\mathbb{R}^{d}$, $i=1,\hdots,K$, for the Minkowski sum, one solves the SDP problem:
\begin{eqnarray}
\underset{\bm{A}_{0},\bm{b}_{0},\tau_{1},\hdots,\tau_{K}}{\text{minimize}}\;\log\det\,\bm{A}_{0}^{-1}	
\label{BoydSDP}
\end{eqnarray}
subject to
	\begin{subequations}
	\begin{alignat}{2}
		&\bm{A}_{0} \succ \bm{0}, \label{posidef}\\
		&\tau_{k} \geq 0, \qquad k=1,\hdots,K, \label{nonneg}\\
		&\begin{bmatrix}
		\bm{E}_{0}^{\top}\bm{A}_{0}\bm{E}_{0} & \bm{E}_{0}^{\top}\bm{b}_{0} & \bm{0}\\
		\bm{b}_{0}^{\top}\bm{E}_{0} & -1 & \bm{b}_{0}^{\top}\\
		\bm{0} & \bm{b}_{0} & -\bm{A}_{0}	
		\end{bmatrix} \!-\! \displaystyle\sum_{k=1}^{K}\tau_{k}\begin{bmatrix}
	\widetilde{\bm{A}}_{k} & \widetilde{\bm{b}}_{k} & \bm{0}\\
	\widetilde{\bm{b}}_{k}^{\top} & c_{k} & \bm{0}\\
	\bm{0} & \bm{0} & \bm{0}
\end{bmatrix}
 \preceq \bm{0}, \label{LMI}
	\end{alignat}	
	\label{Constr}
\end{subequations}
where we let $\bm{E}_{k}$ to be the $d\times dK$ binary matrix that selects the $k$-th vector, $k=1,\hdots,K$, from the vertical stacking of $K$ vectors, each of size $d\times 1$; and for $k=1,\hdots,K$, define
\[\bm{E}_{0} := \displaystyle\sum_{k=1}^{K}\bm{E}_{k}, \quad \widetilde{\bm{A}}_{k} := \bm{E}_{k}^{\top}\bm{A}_{k}\bm{E}_{k}, \quad \widetilde{\bm{b}}_{k}:=\bm{E}_{k}^{\top}\bm{b}_{k}.\]
The argmin pair $(\bm{A}_{0}^{*},\bm{b}_{0}^{*})$ associated with the SDP (\ref{BoydSDP})-(\ref{Constr}), results the optimal ellipsoid
\[\mathcal{E}_{\text{SDP}} := \mathcal{E}\left(\bm{q}_{\text{SDP}},\bm{Q}_{\text{SDP}}\right),\]
where, using (\ref{Ab2Qq}), $\bm{Q}_{\text{SDP}}:=(\bm{A}_{0}^{*})^{-1}$, and $\bm{q}_{\text{SDP}}:=-\bm{Q}_{\text{SDP}}\bm{b}_{0}^{*}$. Our intent is to compare $\mathcal{E}_{\text{SDP}}$ with $\mathcal{E}_{\text{proposed}}$, given by
\[\mathcal{E}_{\text{proposed}} := \mathcal{E}\left(\bm{q}_{\rm{LJ}}, \bm{Q}(\beta_{+})\right),\]
where $\bm{q}_{\rm{LJ}}$ is defined in (\ref{centerMinkSum}), and $\beta^{+}$ is obtained by recursively applying the algorithms proposed in Section IV pairwise to the given set of shape matrices $\bm{Q}_{1}, \hdots, \bm{Q}_{K}$.
 
While the SDP formulation above is applicable for any dimensions, we will see that the algorithms proposed in Section IV help in reducing computational time without sacrificing accuracy. For comparing numerical performance, we implemented both the SDP (via {\texttt{cvx}}) and our proposed algorithms in MATLAB 2016b, on 2.6 GHz Intel Core i5 processor with 8 GB memory.

\subsection{2D Example}
Since Minkowski sum is associative, we implement a recursive version of the root-bracketing followed by bisection algorithm given in Section IV.A.1, that allows us to compute the parameterized MVOE containing the Minkowski sum for a set of $K>2$ ellipses, by applying the proposed method pairwise. 

In Fig. \ref{2Dcomparison}, we show that for $K=4$ randomly generated ellipses, the optimal MVOE $\mathcal{E}_{\text{proposed}}$ computed via the algorithm proposed in Section IV.A.1, agrees with the optimal MVOE $\mathcal{E}_{\text{SDP}}$ obtained by solving (\ref{BoydSDP})-(\ref{Constr}) using {\texttt{cvx}}, with 
\[\vol\left(\mathcal{E}_{\text{proposed}}\right) = 40.1885, \quad \vol\left(\mathcal{E}_{\text{SDP}}\right) = 40.1884.\]
However, the proposed algorithm entails significant savings in computational time compared to the same needed for solving the SDP; in this case
\[t_{\text{proposed}} = 0.009184 \,\:\text{seconds}, \quad t_{\text{SDP}} = 1.513608 \,\:\text{seconds}.\]
The computational time $t_{\text{SDP}}$ reported above does not include the extra processing times needed for setting up the SDP (e.g. construction of matrices $\bm{E}_{k}$, etc.). This order-of-magnitude speed-up was found to be typical for varying $K$ (Fig. \ref{2Dstat}), and is due to the fact that the proposed algorithm computes a custom bracketing range recursively for each pair under consideration.

\begin{figure}[t]
\centering
\includegraphics[width=.5\textwidth]{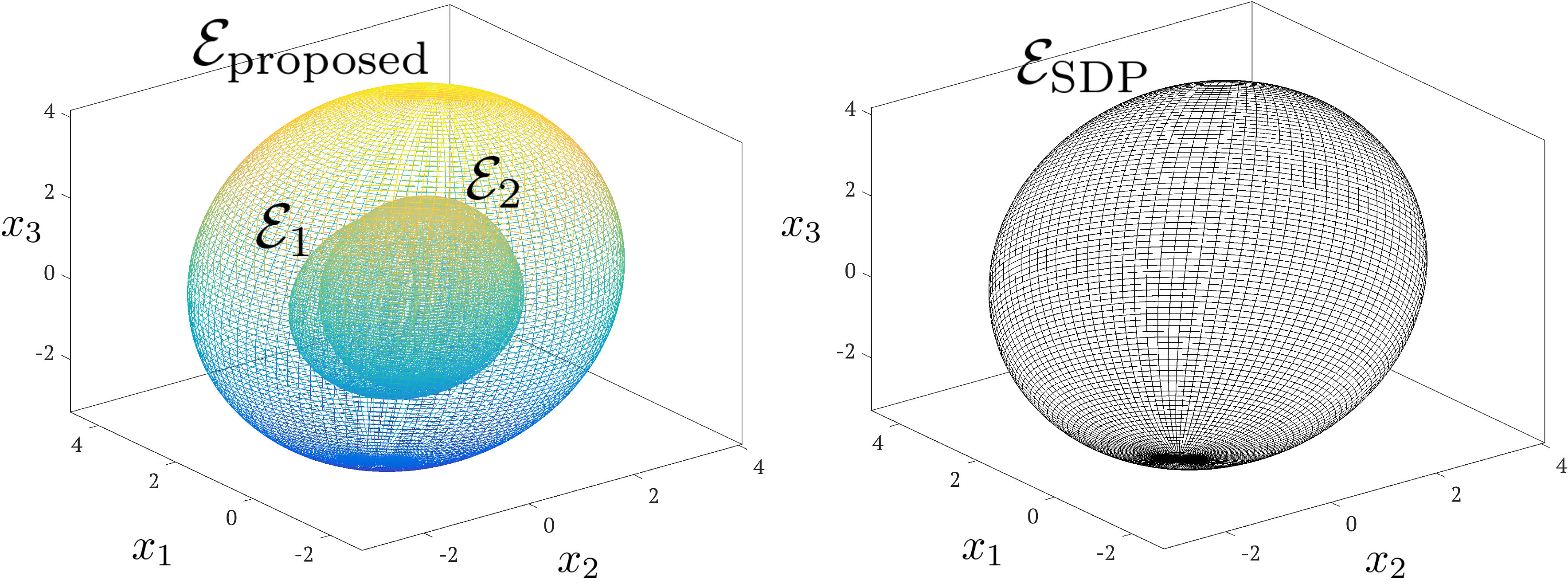}
\vspace*{-0.09in}
\caption{The parameterized MVOE $\mathcal{E}_{\text{proposed}}$ (large ellipsoid on the left subfigure) is shown for the Minkowski sum of K = 2 randomly generated ellipsoids (inner ellipsoids $\mathcal{E}_{1}$ and $\mathcal{E}_{2}$ on the left subfigure) computed using the fixed point recursion (\ref{FixedPointIteration}) described in Section IV.B. The solution $\mathcal{E}_{\text{proposed}}$ matches with the ellipsoid $\mathcal{E}_{\text{SDP}}$ (on the right subfigure) computed by solving (24)-(25) via cvx.}
\vspace*{-0.13in}
\label{3Dcompare}
\end{figure}

\subsection{3D Example}
We only illustrate the solution for $K=2$, since the $K>2$ case can be handled recursively as before. Specifically, for $K=2$ random ellipsoids $\mathcal{E}_{1}$ and $\mathcal{E}_{2}$ in $\mathbb{R}^{3}$ as shown on the left subfigure of Fig. \ref{3Dcompare}, we use the fixed point recursion (\ref{FixedPointIteration}) to compute the parameterized MVOE $\mathcal{E}_{\text{proposed}}$, shown as the large ellipsoid in the left subfigure of Fig. \ref{3Dcompare}. We observe that the $\mathcal{E}_{\text{proposed}}$ thus computed, match with the MVOE obtained by solving the SDP (24)-(25). The SDP solution $\mathcal{E}_{\text{SDP}}$ is shown in the right subfigure of Fig. \ref{3Dcompare}. In this case,
\[\vol\left(\mathcal{E}_{\text{proposed}}\right) = 49.0122, \quad \vol\left(\mathcal{E}_{\text{SDP}}\right) = 49.0121.\]
Again, as in the case of the 2D example, the respective computational times reveal the advantage of the proposed fixed point algorithm:
\[t_{\text{proposed}} = 0.007521 \,\:\text{seconds}, \quad t_{\text{SDP}} = 1.687587 \,\:\text{seconds}.\]
For different problem instances, and varying $K$, we observed computational time statistics similar to Fig. \ref{2Dstat}. We eschew the details for brevity. 


\section{Conclusions}
In this paper, we considered the problem of computing the minimum volume outer ellipsoid (MVOE) of the Minkowski sum of a given set of ellipsoids -- a problem that appears frequently in systems, control and robotics applications. In particular, we focused on computing the so-called inclusion-minimal external parameterized MVOE. We pointed out the equivalence between various forms of such parameterizations appearing in the literature, and provided novel analysis results for the optimality condition. Our analysis led to two new algorithms, which seem to enjoy faster computational time compared to the state-of-the-art semidefinite programming approach of computing the same, without much effect on the numerical quality.


\section*{Acknowledgement}
The author is grateful to John Wayland Bales for suggesting \cite{MSEjwb} the bound (\ref{BracketRootPlanarCase}), and to Suvrit Sra for suggesting \cite{MOsuvrit} the fixed point iteration (\ref{FixedPointIteration}).


\end{document}